\def\draft{n}
\documentclass{amsart}
\usepackage[headings]{fullpage}
\usepackage{amssymb,epic,eepic,epsfig}
\usepackage{pb-diagram,lamsarrow,pb-lams}

\theoremstyle{plain}

\newtheorem{theorem}{Theorem}

\newtheorem{proposition}{Proposition}[section]
\newtheorem{lemma}[proposition]{Lemma}
\newtheorem{corollary}[proposition]{Corollary}

\newtheorem{conjecture}{Conjecture}

\theoremstyle{definition}
\newtheorem{definition}[proposition]{Definition}
\newtheorem{problem}[proposition]{Problem}

\theoremstyle{remark}
\newtheorem{example}[proposition]{Example}

\newtheorem{remark}[proposition]{Remark}

\def\printname#1{
        \if\draft y
                \smash{\makebox[0pt]{\hspace{-0.5in}
                        \raisebox{8pt}{\tt\tiny #1}}}
        \fi
}

\newcommand{\psdraw}[2]
         {\begin{array}{c} \hspace{-1.3mm}
        \raisebox{-4pt}{\epsfig{figure=draws/#1.eps,width=#2}}
        \hspace{-1.9mm}\end{array}}

\newlength{\standardunitlength}
\catcode`\@=11
\long\def\@makecaption#1#2{%
     \vskip 10pt

\setbox\@tempboxa\hbox{
       \small\sf{\bfcaptionfont #1. }\ignorespaces #2}%
     \ifdim \wd\@tempboxa >\captionwidth {%
         \rightskip=\@captionmargin\leftskip=\@captionmargin
         \unhbox\@tempboxa\par}%
       \else
         \hbox to\hsize{\hfil\box\@tempboxa\hfil}%
     \fi}
\font\bfcaptionfont=cmssbx10 scaled \magstephalf
\newdimen\@captionmargin\@captionmargin=2\parindent
\newdimen\captionwidth\captionwidth=\hsize
\catcode`\@=12

\def\lbl#1{\label{#1}\printname{#1}}

\def\BN{\mathbb N}
\def\BZ{\mathbb Z}

\def\BQ{\mathbb Q}
\def\BR{\mathbb R}
\def\BC{\mathbb C}

\def\D{\Delta}

\def\calS{\mathcal S}

\def\a{\alpha}

\def\l{\lambda}
\def\Ga{\Gamma}
\def\S{\Sigma}

\def\ga{\gamma}

\def\la{\langle}
\def\ra{\rangle}

\def\e{\epsilon}
\def\Ga{\Gamma}

\def\b{\beta}

\def\longto{\longrightarrow}

\def\pt{\partial}

\def\Log{\mathrm{Log}}

\def\calA{\mathcal{A}}
\def\calS{\mathcal{S}}

\def\CV{\mathcal{CV}}
\def\ft{\mathfrak{t}}

\def\Tr{\mathrm{Tr}}

\def\gterm{balanced term}
\def\Gterm{Balanced term}
\def\sterm{special term}

\def\hatFT{\widehat{\mathrm{4T}}}
\def\wb2C{\widehat{\b_2(\BC)}}
\def\calC{\mathcal{C}}

\def\CV{\mathrm{CV}}

\begin{document}


\title[An ansatz for the asymptotics of hypergeometric multisums]{
An ansatz for the asymptotics of hypergeometric multisums}
\author{Stavros Garoufalidis}
\address{School of Mathematics \\
         Georgia Institute of Technology \\
         Atlanta, GA 30332-0160, USA \\ 
         {\tt http://www.math.gatech} \newline {\tt .edu/$\sim$stavros } }
\email{stavros@math.gatech.edu}

\thanks{The author was supported in part by NSF. \\
\newline
1991 {\em Mathematics Classification.} Primary 57N10. Secondary 57M25.
\newline
{\em Key words and phrases: holonomic functions, regular holonomic functions,
special hypergeometric terms, WZ algorithm, diagonals of rational functions,
asymptotic expansions, transseries, Laplace transform, Hadamard product,
Bethe ansatz, singularities, additive $K$-theory, regulators,
Stirling formula, algebraic combinatorics, polytopes, Newton polytopes,
$G$-functions.
}
}

\date{February 14, 2008 }


\begin{abstract}
Sequences that are defined by multisums of hypergeometric terms 
with compact support occur frequently in enumeration problems of 
combinatorics, algebraic geometry and perturbative quantum field theory. 
The standard recipe to study the asymptotic expansion of such sequences
is to find a recurrence satisfied by them, convert it into a differential
equation satisfied by their generating series, and analyze the singulatiries
in the complex plane. We propose a shortcut by constructing directly 
from the structure of the hypergeometric term a finite set, for which
we conjecture (and in some cases prove)
that it contains all the singularities of the generating series.
Our construction of this finite set is given by the solution set
of a balanced system of polynomial equations of a rather special form, 
reminiscent of the Bethe ansatz. The finite set can also be identified
with the set of critical values of a potential function, as well as
with the evaluation of elements of an additive $K$-theory group by a 
regulator function. We give a proof of our conjecture in some special cases, 
and we illustrate our results with numerous examples.
\end{abstract}

\maketitle

\tableofcontents

\section{Introduction}
\lbl{sec.intro}

\subsection{The problem}
\lbl{sub.problem}

The problem considered here is the following: given a balanced
hypergeometric term
$\ft_{n,k_1,\dots,k_r}$ with compact support for each $n \in \BN$, we wish
to find an asymptotic expansion of the sequence 

\begin{equation}
\lbl{eq.aftn}
a_{\ft,n}=\sum_{(k_1,\dots,k_r) \in \BZ^r} \ft_{n,k_1,\dots,k_r}.
\end{equation}
Such sequences occur frequently
in enumeration problems of combinatorics, algebraic geometry and perturbative
quantum field theory; see \cite{St,FS,KM}.
The standard recipe for this is to find a recurrence satisfied by $(a_n)$, 
convert it into a differential equation satisfied by the generating series
$$
G_{\ft}(z)=\sum_{n=0}^\infty a_{\ft,n} z^n
$$ 
and analyze the singulatiries of its analytic continuation 
in the complex plane. We propose a shortcut by constructing directly 
from the structure of the hypergeometric term $\ft_{n,k_1,\dots,k_r}$ a finite set
$S_{\ft}$, for which we conjecture (and in some cases prove)
that it contains all the singularities of $G(z)$.
Our construction of $S_{\ft}$ is given by the solution set
of a balanced system of polynomial equations of a rather special form, 
reminiscent of the Bethe ansatz. $S_{\ft}$ can also be identified
with the set of critical values of a potential function, as well as
with the evaluation of elements of an additive $K$-theory group by a 
regulator function. We give a proof of our conjecture in some special cases, 
and we illustrate our results with numerous examples.

\subsection{Existence of asymptotic expansions}
\lbl{sub.motivation}

A general existence theorem for asymptotic expansions of sequences discussed
above was recently given in \cite{Ga4}. To phrase it, we need to recall 
what is a $G$-function in the sense of Siegel; \cite{Si}.

\begin{definition}
\lbl{def.Gfunction}
We say that series $G(z)=\sum_{n=0}^\infty a_n z^n$ is a {\em $G$-function}
if 
\begin{itemize}
\item[(a)]
the coefficients $a_n$ are algebraic numbers, and 
\item[(b)]
there exists a constant $C>0$ so that for every $n \in \BN$
the absolute value of every conjugate of $a_n$ is less than or equal to 
$C^n$, and
\item[(c)]
the common denominator of $a_0,\dots, a_n$ is less than or equal 
to $C^n$.
\item[(d)]
$G(z)$ is holonomic, i.e., it satisfies a linear differential equation
with coefficients polynomials in $z$.
\end{itemize}
\end{definition}

The main result of \cite{Ga4} is the following theorem.

\begin{theorem}
\lbl{thm.G}\cite[Thm.3]{Ga4}
For every  \gterm\ $\ft$, the generating series $G_{\ft}(z)$
is a $G$-function. 
\end{theorem}
Using the fact that the local monodromy of a $G$-function around a singular
point is {\em quasi-unipotent} (see \cite{Ka,An2,CC}), an elementary 
application of Cauchy's theorem
implies the following corollary; see \cite{Ga4}, \cite{Ju} and 
\cite[Sec.7]{CG1}.

\begin{corollary}
\lbl{cor.assexp}
If $G(z)=\sum_{n=0}^\infty a_n z^n$ is a $G$-function, then $(a_n)$ has a 
{\em transseries expansion}, that is an expansion of the form
\begin{equation}
\lbl{eq.trans}
a_n \sim \sum_{\l \in \S} \l^{-n} n^{\a_{\l}} (\log n)^{\b_{\l}} \sum_{s=0}^\infty 
\frac{c_{\l,s}}{n^s}
\end{equation}
where $\S$ is the set of singularities of $G(z)$, $\a_{\l} \in \BQ$,
$\b_{\l} \in \BN$, and $c_{\l,s} \in \BC$. In addition, $\S$ is a finite
set of algebraic numbers, and generates a number field $E=\BQ(\S)$.
\end{corollary}

\subsection{Computation of asymptotic expansions}
\lbl{sub.compexp}

Theorem \ref{thm.G} and its Corollary \ref{cor.assexp} are not constructive.
The usual way for computing the asymptotic expansion for sequences 
$(a_{\ft,n})$
of the form \eqref{eq.aftn} is to find a linear reccurence, and convert
it into a differential equation for the generating series $G_{\ft}(z)$.
The singularities of $G_{\ft}(z)$ are easily located from the roots of 
coefficient of the leading derivative of the ODE. 
This approach is taken by Wimp-Zeilberger, following Birkhoff-Trjitzinsky;
see \cite{BT,WZ2} and also \cite{Ni}. A reccurence for a multisum sequence
$(a_{\ft,n})$ follows from Wilf-Zeilberger's constructive theorem, and 
its computer implementation; see \cite{Z,WZ1,PWZ,PR1,PR2}. Although 
constructive, these algorithms are impractical for multisums with, say, 
more than three summation variables. 

On the other hand, it seems wasteful to compute an ODE for $G_{\ft}(z)$,
and then discard all but a small part of it in order to determine the
singularities $\S_{\ft}$ of $G_{\ft}(z)$.

The main result of the paper is a construction of a finite set $S_{\ft}$
of algebraic numbers directly from the summand $\ft_{n,k_1,\dots,k_r}$, which 
we conjecture that it includes the set $\S_{\ft}$. We give a proof of our
conjecture in some special cases, as well as supporting examples.

Our definition of the set $\S_{\ft}$ is reminiscent of the Bethe ansatz, 
and is related to critical values of potential functions and
additive $K$-theory.

Before we formulate our conjecture let us give an 
instructive example.

\begin{example}
\lbl{ex.0}
Consider the {\em Apery sequence} $(a_n)$ defined by
\begin{equation}
\lbl{eq.aperya}
a_n=\sum_{k=0}^n \binom{n}{k}^2 \binom{n+k}{k}^2.
\end{equation}
It turns out that $(a_n)$ 
satisfies a linear recursion relation with coefficients in $\BQ[n]$
(see \cite[p.174]{WZ2})

\begin{equation}
\lbl{eq.apery2}
(n+2)^3 a_{n+2} -
(2n+3)(17n^2+51n+39) a_{n+1}+(n+1)^3 a_n=0
\end{equation}
for all $n \in \BN$, with initial conditions $a_0=1, a_1=5$.
It follows that $G(z)$ is {\em holonomic}; i.e., it satisfies a linear ODE
with coefficients in $\BQ[z]$:

\begin{equation}
\lbl{eq.apery3}
z^2(z^2-34z+1) G'''(z)+3z(2z^2-51z+1)G''(z)+(7z^2-112z+1)G'(z)
+(z-5) G(z)=0
\end{equation}
with initial conditions $G(0)=1$, $G'(0)=5$, $G''(0)=146$.

This implies that the possible singularities of $G(z)$ are the roots of the
equation:

\begin{equation}
\lbl{eq.aperyE}
z^2(z^2-34 z +1)=0.
\end{equation}
Thus, $G(z)$ has analytic continuation as a multivalued 
function in $\BC\setminus\{0,17+12\sqrt{2},17-12\sqrt{2}\}$.
Since the Taylor series coefficients of $G(z)$ at $z=0$ are positive integers,
and $G(z)$ is analytic at $z=0$, it follows that $G(z)$ has a singularity
inside the punctured unit disk. Thus, $G(z)$ is singular at $17-12\sqrt{2}$.
By Galois invariance, it is also singular at $17+12\sqrt{2}$.
The proof of Corollary \ref{cor.assexp} implies that $(a_n)$ 
has an asymptotic expansion of the form:
\begin{equation}
\lbl{eq.aperyass}
a_n \sim (17+12 \sqrt{2})^n n^{-3/2} \sum_{s=0}^\infty \frac{c_{1s}}{n^s}
+ (17-12 \sqrt{2})^n n^{-3/2} \sum_{s=0}^\infty \frac{c_{2s}}{n^s}
\end{equation}
for some constants $c_{1s}, c_{2s} \in \BC$ with $c_{10} c_{20} \neq 0$. 
A final calculation shows that 
\begin{equation*}
c_{10} = \frac{1}{\pi^{3/2}} \frac{3+2 \sqrt{2}}{4 \sqrt[4]{2}},
\qquad c_{20} = \frac{1}{\pi^{3/2}} \frac{3-2 \sqrt{2}}{4 \sqrt[4]{2}}.
\end{equation*}
\end{example}
Our paper gives an ansatz that quickly produces the numbers $17\pm 2 \sqrt{2}$
given the expression \eqref{eq.aperya} of $(a_n)$, bypassing Equations
\eqref{eq.apery2} and \eqref{eq.apery3}. This is explained in Section
\ref{sub.apery}. In a forthcoming publication we will explain how to
compute the {\em Stokes constants} $c_{s0}$ for $s=0$, $1$ in terms of the 
expression \eqref{eq.aperya}.

\subsection{Hypergeometric terms}
\lbl{sub.hyper}

We have already mentioned multisums of balanced hypergeometric terms.
Let us define what those are.

\begin{definition}
\lbl{def.hyperg}
An $r$-dimensional 
{\em balanced hypergeometric term} $\ft_{n,k}$ (in short, {\em \gterm },
also denoted by $\ft$) in variables $(n,k)$, where $n \in \BN$ and 
$k=(k_1,\dots,k_r) \in \BZ^r$, is an expression of the form:
\begin{equation}
\lbl{eq.defterm}
\ft_{n,k}=C_0^n \prod_{i=1}^r C_i^{r_i} \prod_{j=1}^J A_j(n,k)!^{\e_j}
\end{equation}
where $C_i$ are algebraic numbers for $i=0,\dots,r$, 
$\e_j=\pm 1$ for $j=1,\dots,J$, and $A_j$ are integral 
linear forms in $(n,k)$ that satisfy the {\em balance condition}:
\begin{equation}
\lbl{eq.Ajsum}
\sum_{j=1}^J \e_j A_j=0.
\end{equation}
\end{definition}

\begin{remark}
\lbl{rem.encode}
An alternative way of encoding a \gterm\ is to record the vector 
$(C_0,\dots,C_r)$, and the $J \times (r+2)$ matrix of the coefficients
of the linear forms $A_j(n,k)$, and the signs $\e_j$ for $j=1,\dots,r$.
\end{remark}

\subsection{\Gterm s, generating series and singularities}
\lbl{sub.gseries}

Given a \gterm\ $\ft$, we will assign a sequence $(a_{\ft,n})$
and a corresponding generating series $G_{\ft}(z)$
to a \gterm\ $\ft$, and will study the set $\S_{\ft}$ of singularities of the
analytic continuation of $G_{\ft}(z)$ in the complex plane. 

Let us introduce some useful notation. 
Given a linear form $A(n,k)$ in variables $(n,k)$ where $k=(k_1,\dots,k_r)$,
and $i=0,\dots,r$, let us define

\begin{equation}
\lbl{eq.vi}
v_i(A)=a_i, \qquad v_0(A)=a_0, \qquad 
\text{where} \qquad A(n,k)=a_0 n+\sum_{i=1}^r a_i k_i.
\end{equation}
For $w=(w_1,\dots,w_r)$ we define:

\begin{equation}
\lbl{eq.Aw}
A(w)=a_0+\sum_{i=1}^r a_i w_i.
\end{equation}

\begin{definition}
\lbl{def.polytope}
Given a \gterm\ $\ft$ as in \eqref{eq.defterm}, define its
{\em Newton polytope} $P_{\ft}$ by:
\begin{equation}
\lbl{eq.Pft}
P_{\ft}=\{ w \in \BR^r \,|  \, A_j(w) \geq 0 \,\text{for}\, 
j=1,\dots, J\}
\subset \BR^r .
\end{equation}
\end{definition} 
We will assume that $P_{\ft}$ is a compact
rational convex polytope in $\BR^r$ with non-empty interior.
It follows that for every $n \in \BN$ we have:
\begin{equation}
\lbl{eq.polytope}
\mathrm{support}(\ft_{n,k})= n P_{\ft} \cap \BZ^{r}.
\end{equation}

\begin{definition}
\lbl{def.aG}
Given a \gterm\ $\ft$ consider the sequence:
\begin{equation}
\lbl{eq.an}
a_{\ft,n}=\sum_{k \in n P_{\ft} \cap \BZ^r} \ft_{n,k}
\end{equation}
(the sum is finite for every $n \in \BN$)
and the corresponding generating function:
\begin{equation}
\lbl{eq.Gz}
G_{\ft}(z)=\sum_{n=0}^\infty a_{\ft,n} z^n \in \bar\BQ[[z]].
\end{equation}
Here $\bar\BQ$ denote the field of algebraic numbers.
Let $\S_{\ft}$ denote the finite set of singularities of $G_{\ft}$,
and $E_{\ft}=\BQ(\S_{\ft})$ denote the corresponding number field,
following Corollary \ref{cor.assexp}. 
\end{definition}

\begin{remark}
\lbl{rem.tG}
Notice that $\ft$ determines $G_{\ft}$ but {\em not} vice-versa.
Indeed, there are nontrivial identities among multisums of \gterm s.
Knowing a complete set of such identities would be very useful in
constructing invariants of knotted objects, as well as in
understanding relations among periods; see \cite{KZ}.
\end{remark}

\begin{remark}
\lbl{rem.balance}
The balance condition of Equation \eqref{eq.Ajsum} is imposed so that
for every \gterm\ $\ft$ the corresponding sequence $(a_{\ft,n})$ grows at most 
exponentially. This follows from Stirling's formula (see Corollary 
\ref{cor.app}) and it implies that the power series $G_{\ft}(z)$ is 
the germ of an analytic function at $z=0$.
Given a proper hypergeometric term $\ft_{n,k}$ in the sense of \cite{WZ1},
we can find $\a \in \BN$ and $\e=\pm 1$
so that $\ft_{n,k} (\a n)!^{\e}$ is a \gterm .
\end{remark}

\subsection{The definition of 
$S_{\ft}$ and $K_{\ft}$}
\lbl{sub.defhat}

Let us observe that if $\ft$ is a \gterm\ and $\D$ is a face of its
Newton polytope $P_{\ft}$, then $\ft|_{\D}$ is also a \gterm . 

\begin{definition}
\lbl{def.vareq}
Given a \gterm\ $\ft$ as in Equation \eqref{eq.defterm}
consider the following system of {\em Variational Equations}:
\begin{equation}
\lbl{eq.vara}
C_i \prod_{j=1}^J A_j(w)^{\e_j  v_i(A_j)}
=1 \qquad \text{for} \,\, i=1,\dots, r 
\end{equation}
in the variables $w=(w_1,\dots,w_r)$.
Let $X_{\ft}$ denote the set of complex solutions of \eqref{eq.vara},
with the convention that when $r=0$ we set $X_{\ft}=\{0\}$, and define
\begin{eqnarray}
\lbl{eq.StD}
\CV_{\ft} &=& \{ C_0^{-1} \prod_{j: A_j(w) \neq 0} A_j(w)^{-\e_j v_0(A_j)} \,\, | \,\,
w \in X_{\ft} \} 
\\
\lbl{eq.Shat}
S_{\ft}&=&\{0\} \cup \cup_{\D \,\mathrm{face}\,\,\mathrm{of} \, P_{\ft}} \CV_{\ft|_\D} \\
\lbl{eq.Khat}
K_{\ft}&=&\BQ(S_{\ft}).
\end{eqnarray}
\end{definition}

\begin{remark}
\lbl{rem.CVset}
There are two different incarnations of the set $\CV_{\ft}$: it coincides
with 
\begin{itemize}
\item[(a)]
the set of critical values of a potential function; see Theorem
\ref{thm.critV}. 
\item[(b)]
the evaluation of elements of an additive $K$-theory group under the
entropy regulator function; see Theorem \ref{thm.K}.
\end{itemize}
\end{remark}

It is unknown to the author whether $X_{\ft}$ is always a finite set. 
Nevertheless, $S_{\ft}$ is always a finite subset of $\bar\BQ$; see 
Theorem \ref{thm.critV}. 
Equations \eqref{eq.vara} are reminiscent of the {\em Bethe ansatz}.

\subsection{The conjecture}
\lbl{sub.statement}
Section \ref{sub.gseries} constructs a map:
$$
\text{\Gterm s $\ft$} \longto (E_{\ft},\S_{\ft})
$$
via generating series $G_{\ft}(z)$ and their singularities, where $E_{\ft}$
is a number field and $\S_{\ft}$ is a finite subset of $E_{\ft}$.
Section \ref{sub.defhat} constructs a map:
$$
\text{\Gterm s $\ft$} \longto (K_{\ft},S_{\ft})
$$
via solutions of polynomial equations.
We are now ready to formulate our main conjecture.

\begin{conjecture}
\lbl{conj.1}
For every \gterm\ $\ft$ we have: $\S_{\ft} \subset S_{\ft}$
and consequently, $E_{\ft} \subset K_{\ft}$.
\end{conjecture}

\subsection{Partial results}
\lbl{sub.partial}

Conjecture \ref{conj.1} is known to hold in the following cases:

\begin{itemize}
\item[(a)] For $0$-dimensional \gterm s: see Theorem \ref{thm.kempty}.
\item[(b)] For positive \sterm s: see Theorem \ref{thm.1}.
\item[(c)] For $1$-dimensional \gterm s, see \cite{Ga3}.
\end{itemize}
Since the finite sets $S_{\ft}$ and $\S_{\ft}$ that appear in Conjecture
\ref{conj.1} are in principle computable (as explained in Section 
\ref{sub.compexp}), one may try to check random examples. We give some
evidence in Section \ref{sec.examples}. We refer the reader to \cite{GV}
for an interesting class of $1$-dimensional examples related to
$6j$-symbols, and of interest to 
atomic physics and low dimensional topology.

The following proposition follows
from the classical fact concerning singularities of hypergeometric
series; see for example \cite[Sec.5]{O} and \cite{No}.

\begin{theorem}
\lbl{thm.kempty}
Suppose that $\ft_{n,k}$ is 0-dimensional \gterm\ as in \eqref{eq.defterm}, 
with $k=\emptyset$. Then $G_{\ft}(z)$ is a hypergeometric series and 
Conjecture \ref{conj.1} holds.
\end{theorem}

When $\ft$ is positive dimensional, the generating series $G_{\ft}(z)$
is no longer hypergeometric in general. To state our next result, recall
that a finite sybset $S \subset \bar\BQ$ of algebraic numbers is 
{\em irreducible} over $\BQ$ if the Galois group $\text{Gal}(\bar\BQ,\BQ)$
acrs transitively on $S$.

\begin{definition}
\lbl{def.shyperg}
A {\em special hypergeometric term} $\ft_{n,k}$ (in short, {\em \sterm })
is an expression of the form:
\begin{equation}
\lbl{eq.tnk}
\ft_{n,k}=C_0^n \prod_{i=1}^r C_i^{r_i} \prod_{j=1}^J \binom{B_j(n,k)}{C_j(n,k)}.
\end{equation}
where $C \in \BQ$, $L$ and $B_j$ and $C_j$ are integral linear forms
in $(n,k)$. 
We will assume that for every $n \in \BN$, the support of $\ft_{n,k}=0$ as
a function of $k \in \BZ^r$ is finite. We will call such a term {\em positive}
if $C_i>0$ for $i=0,\dots,r$.
\end{definition}

\begin{lemma}
\lbl{lem.specgen}
\rm{(a)}
A \gterm\ is the ratio of two \sterm s. In other words, it
can always be written in the form:
\begin{equation}
\lbl{eq.tnk3}
\ft_{n,k}=C^{L(n,k)} \prod_{j=1}^s \binom{B_j(n,k)}{C_j(n,k)}^{\e_j}
\end{equation}
for some integral linear forms $B_j,C_j$ and signs $\e_j$.  
\newline
\rm{(b)} The set of \sterm s is an abelian monoid with respect to 
multiplication, whose corresponding abelian group is the 
set of \gterm s.
\end{lemma}
The proof of (a) follows from writing a \gterm\ in the form:
$$
\ft_{n,k}=C_0^n \prod_{i=1}^r C_i^{r_i} 
\frac{\prod_{j: \e_j=1}^J A_j(n,k)!}{A(n,k)!}
\frac{A(n,k)!}{\prod_{j: \e_j=-1}^J A_j(n,k)!}
$$
where
$$
A(n,k)=\sum_{j: \e_j=1}^J A_j(n,k)=\sum_{j: \e_j=-1}^J A_j(n,k).
$$

To illustrate part (a) of the above lemma for $0$-dimensional \gterm s,
we have:
$$
\frac{(30n)! n!}{(16n)! (10n)! (5n)!}=
\binom{30n}{16n} \binom{14n}{10n} \binom{5n}{4n}^{-1}.
$$
The above identity also shows that if a \gterm\ takes integer values, 
it need not be a \sterm . This phenomenon was studied by Rodriguez-Villegas;
see \cite{R-V}.

\begin{theorem}
\lbl{thm.1}
Fix a positive \sterm\ $\ft_{n,k}$ such
that $\S_{\ft}\setminus\{0\}$ is irreducible over $\BQ$. Then, 
$\S_{\ft} \subset \CV_{\ft} \subset S_{\ft}$ and 
Conjecture \ref{conj.1} holds. 
\end{theorem}

In a forthcoming publication we will give a proof of Conjecture \ref{conj.1}
for 1-dimensional \gterm s; see \cite{Ga3}.
Let us end this section with an inverse type (or geometric realization)
problem.

\begin{problem}
\lbl{prob.2}
Given $\l \in \bar\BQ$, does there exist a special term $\ft$ so that 
$\l \in S_{\ft}$?
\end{problem}

\subsection{Laurent polynomials: a source of \sterm s}
\lbl{sub.laurent}

This section, which is of independent interest, associates an \sterm\
$\ft_F$ to a Laurent polynomial $F$ with the property
that the generating series $G_{\ft_F}(z)$ is identified with the 
trace of the resolvant $R_F(z)$ of $F$. 
Combined with Theorem \ref{thm.G}, this implies that $R_F(z)$ is
a $G$-function. 

If $F \in M_N(\bar\BQ[x_1^{\pm 1},\dots, x_d^{\pm 1}])$ is a square matrix
of size $N$ with entries Laurent polynomials in $d$ variables, let 
$\Tr(F)$ denote the {\em constant term} of its usual trace.
The {\em moment generating series} of $F$ is the power series

\begin{equation}
\lbl{eq.GF}
G_F(z)=\sum_{n=0}^\infty \Tr(F^n) z^n.
\end{equation}

\begin{theorem}
\lbl{thm.3}
\rm{(a)}
For every $F \in \bar\BQ[x_1^{\pm 1},\dots, x_d^{\pm 1}]$ there exists a \sterm\
$\ft_F$ so that 
\begin{equation}
\lbl{eq.Pta}
G_F(z)=G_{\ft_F}(z).
\end{equation}
Consequently, $G_F(z)$ is a $G$-function.
\newline
\rm{(b)} The Newton polytope $P_{\ft_F}$ depends is a combinatorial simplex
which depends only on the monomials
that appear in $F$ and not on their coefficients.
\newline
\rm{(c)}
For every $F \in M_N(\bar\BQ[x_1^{\pm 1},\dots, x_d^{\pm 1}])$, $G_F(z)$ is a 
$G$-function.
\end{theorem}
We thank C. Sabbah for providing an independent proof of part (c) when $N=1$ 
using the regularity of the {\em Gauss-Manin connection}. Compare also
with \cite{DvK}.

\subsection{Plan of the proof}
\lbl{sub.plan}

In Section \ref{sec.potential}, we introduce a potential function 
associated to 
a \gterm\ $\ft$ and we show that the set of its critical values coincides
with the set $S_{\ft}$ that features in Conjecture \ref{conj.1}.
This also implies that $S_{\ft}$ is finite.

In Section \ref{sec.stirling} we assign elements of an extended additive
$K$-theory group to a \gterm\ $ft$, and we show that the set of 
their values (under the entropy regulator map) coincides with the set $S_{\ft}$  
that features in Conjecture \ref{conj.1}.

In Section \ref{sec.pfthm1} we give a proof of Theorems \ref{thm.1}
(using results from hypergeometric functions) 
and \ref{thm.kempty} (using an application of Laplace's method), 
which are partial case of our Conjecture \ref{conj.1}. 

In Section \ref{sec.examples} we give several examples that illustrate
Conjecture \ref{conj.1}.

In Section \ref{sec.laurentp} we study a special case of Conjecture 
\ref{conj.1}, with input a Laurent polynomial in many commuting variables.

\subsection{Acknowledgement}
The author wishes to thank Y. Andr\'e, N. Katz, M. Kontsevich,
J. Pommersheim, C. Sabbah, and especially D. Zeilberger for many enlightening 
conversations, R.I. van der Veen for a careful reading of the 
manuscript, and the anonymous referee for comments that improved the 
exposition. An early
version of the paper was presented in an Experimental Mathematics Seminar 
in Rutgers in the spring of 2007. The author wishes to thank D. Zeilberger
for his hospitality.

\section{\Gterm s and potential functions}
\lbl{sec.potential}

\subsection{The Stirling formula and potential functions}
\lbl{sub.potential}

As a motivation of a potential function associated to a \gterm\ , recall 
{\em Stirling formula}, which computes the 
asymptotic expansion of $n!$ (see \cite{O}):

\begin{equation}
\lbl{eq.stirling}
\log n! \sim n \log n -n + \frac{1}{2} \log n + \frac{1}{2} \log(2 \pi)
+O\left(\frac{1}{n}\right),
\end{equation}
For $x>0$, we can define $x!=\Ga(x+1)$. The Stirling formula implies that
for $a >0$ fixed and $n \in \BN$ large, we have:

\begin{equation}
\lbl{eq.stan}
\log (an)! = (n \log n -n) a  + a \log(a) + O\left(\frac{\log n}{n}\right)
\end{equation}
The next corollary motivates our definition of the potential function.

\begin{corollary}
\lbl{cor.app}
For every \gterm\ $\ft$ as in \eqref{eq.defterm}
and every $w$ in the interior of $P_{\ft}$ we have:
\begin{equation}
\lbl{eq.binomialp}
\ft_{n,nw} = e^{n V_{\ft}(w)+O\left(\frac{\log n}{n}\right)}.
\end{equation}
where the potential function $V_{\ft}$ is defined below.
\end{corollary}

\begin{definition}
\lbl{def.potential}
Given a \gterm\ $\ft$ as in \eqref{eq.defterm} 
define its corresponding {\em potential function} $V_{\ft}$ by:
\begin{equation}
\lbl{eq.potentialf}
V_{\ft}(w)=C(w)+
\sum_{j=1}^J \e_j A_j(w) \log(A_j(w))
\end{equation}
where $w=(w_1,\dots,w_r)$, 
\begin{equation}
\lbl{eq.Cnk}
C(n,k)=\log C_0 \cdot n + \log C_1 \cdot k_1 + \dots \log C_r \cdot k_r,
\end{equation}
and 
\begin{equation}
\lbl{eq.Cw}
C(w)=\log C_0 + \log C_1 w_1 + \dots \log C_r w_r.
\end{equation}
$V_{\ft}$ is a multivalued analytic function on the complement of the 
{\em linear hyperplane arrangement} $\BC^{r}\setminus\calA_{\ft}$, where
\begin{equation}
\lbl{eq.arr}
\calA_{\ft}=\{ w\in \BC^{r} \,| \, \prod_{j=1}^J A_j(w)=0 \}.
\end{equation}
\end{definition}
In fact, let
\begin{equation}
\lbl{eq.varpi2}
\varpi: \hat{\BC}^{r} \longto \BC^{r}\setminus\calA_{\ft}
\end{equation}
denote the universal abelian cover of $\BC^{r}\setminus\calA_{\ft}$.

The next theorem relates the critical points and critical values of $V_{\ft}$
with the sets $X_{\ft}$ and $\CV_{\ft}$ from Definition \ref{def.vareq}.

\begin{theorem}
\lbl{thm.critV}
\rm{(a)} For every \gterm\ $\ft$, $\varpi^{-1}(X_{\ft})$ 
coincides with the set of critical points of $V_{\ft}$.
\newline
\rm{(b)} If $w$ is a critical point of $V_{\ft}$, then
\begin{equation}
\lbl{eq.lambdar}
e^{-V_{\ft}(w)}= C^{-v_0(L)} \prod_{j=1}^J A_j(w)^{-\e_j v_0(A_j)}=
C^{-v_0(L)} \prod_{j: A_j(w) \neq 0} A_j(w)^{-\e_j v_0(A_j)}.
\end{equation}
Thus $\CV_{\ft}$ coincides with the exponential of the set of the
negatives of the critical values of $V_{\ft}$.
\newline
\rm{(c)} For every face $\D$ of the Newton polytope of $\ft$ we have:
$$
V_{\left(\ft|_{\D}\right)}=\left(V_{\ft}\right)|_{\D}.
$$
\rm{(d)} $S_{\ft}$ is a finite subset of $\bar{\BQ}$ and 
$K_{\ft}$ is a number field.
\end{theorem}

\subsection{Proof of Theorem \ref{thm.critV}}
\lbl{sub.thmcritV}

Let us fix a \gterm\ $\ft$ as in \eqref{eq.defterm}
and a face $\D$ of its Newton polytope $P_{\ft}$. Without loss of
generality, assume that $\D=P_{\ft}$. Since 

\begin{equation}
\lbl{eq.derspecial}
\frac{d}{dx}(x \log(x))=\log(x)+1
\end{equation}
it follows that for every $i=1,\dots,r$ and every $j=1,\dots,J$
we have:

\begin{equation}
\frac{\pt}{\pt w_i} A_j(w) \log(A_j(w))=v_i(A_j) \log(A_j(w)) + v_i(A_j).
\end{equation}
Adding up with respect to $j$, using the balancing condition
of Equation \eqref{eq.Ajsum}, and the notation of Section \ref{sub.gseries}
applied to the linear form of Equation \eqref{eq.Cnk}, it follows that

\begin{eqnarray*}
\frac{\pt}{\pt w_i} V_{\ft}(w)&=&  \log C_i +
\sum_{j=1}^J \e_j v_i(A_j) \log(A_j(w)) + \sum_{j=1}^J \e_j v_i(A_j) \\
&=& \log C_i +
\sum_{j=1}^J \e_j v_i(A_j) \log(A_j(w)).
\end{eqnarray*}
This proves that the critical points $w=(w_1,\dots,w_r)$ of $V_{\ft}$ 
are the solutions
to the following system of {\em Logarithmic Variational Equations}:
\begin{equation}
\lbl{eq.varlog}
 \log C_i+
\sum_{j=1}^J \e_j v_i(A_j) \log(A_j(w)) \in \BZ(1) \qquad 
\text{for} \,\,i=1,\dots,r 
\end{equation}
where, for a subgroup $K$ of $(\BC,+)$ and an integer $n \in \BZ$, we define 

\begin{equation}
\lbl{eq.Kn}
K(n)=(2 \pi i)^n K.
\end{equation}
Exponentiating, it follows that $w$ satisfies
the Variational Equations \eqref{eq.vara}, and concludes the proof of part (a).

For part (b), we will show that if $w$ is a critical point of $V_{\ft}$, 
the corresponding critical value is given by:
\begin{equation}
\lbl{eq.criticalvV}
V_{\ft}(w)= \log C_0 + \sum_{j=1}^J \e_j v_0(A_j) \log(A_j(w)) \in \BC/\BZ(1).
\end{equation}
Exponentiating, we deduce the first equality of Equation \eqref{eq.lambdar}.
The second equality follows from the fact that $A_j(w) \neq 0$
for all critical points $w$ of $V_{\ft}$.

To show \eqref{eq.criticalvV}, 
observe that for any linear form $A(n,k)$ we have:
$$
A(w)=v_0(A)w_0 + \sum_{i=1}^r v_i(A) w_i.
$$
Suppose that $w$ satisfies the Logarithmic
Variational Equations \eqref{eq.varlog}.
Using the definition of the potential function, and collecting terms
with respect to $w_1,\dots,w_r$ it follows that

\begin{eqnarray*}
V_{\ft}(w) &=&  \log C_0 + \sum_{j=1}^J \e_j v_0(A_j) \log(A_j(w)) 
+ \sum_{i=1}^r w_i
\left(\log C_i  + \sum_{j=1}^J \e_j v_i(A_j) \log(A_j(w))\right) \\
 &=&  \log C_0  + \sum_{j=1}^J \e_j v_0(A_j) \log(A_j(w)) .
\end{eqnarray*}
This concludes part (b). Part (c) follows from set-theoretic 
considerations, and part (d) follows from the following facts:

\begin{itemize}
\item[(i)]
an analytic function is constant on
each component of its set of critical points,
\item[(ii)]
the set of critical points are the complex points of an affine variety
defined over $\BQ$ by \eqref{eq.vara},
\item[(iii)]
every affine variety has finitely many connected components.
\end{itemize}
This concludes the proof of Theorem \ref{thm.critV}.
\qed

\section{\Gterm s, the entropy function and additive $K$-theory}
\lbl{sec.stirling}

In this section we will assign elements of an extended additive $K$-theory
group to a \gterm , and using them, we will identify our finite set $\CV_{\ft}$
from Definition \eqref{def.vareq} 
with the values of the constructed elements under a regulator map;
see Theorem \ref{thm.K}.

\subsection{A brief review of the entropy function and additive $K$-theory}
\lbl{sub.entropy}

In this section we will give a brief summary of an extended version of
additive $K$-theory and the entropy function following \cite{Ga2}, and
motivated by \cite{Ga1}. This section is independent of the rest of the 
paper, and may be skipped at first reading.

\begin{definition}
\lbl{def.entropy}
Consider the {\em entropy function} $\Phi$, defined by:
\begin{equation}
\lbl{eq.entropy}
\Phi(x)=-x \log (x)-(1-x) \log(1-x).
\end{equation}
for $x \in (0,1)$.
\end{definition}

$\Phi(x)$ is a multivalued analytic function on $\BC\setminus\{0,1\}$, 
given by the double integral of a rational function as follows from:

\begin{equation}
\lbl{eq.doubleint}
\Phi''(x)=-\frac{1}{x}-\frac{1}{1-x}.
\end{equation}
For a detailed description of the analytic continuation of $\Phi$,
we refer the reader to \cite{Ga2}. 
Let $\hat{\BC}$ denote the {\em universal abelian cover} of 
$\BC^{**}$. In \cite[Sec.1.3]{Ga2} we show that $\Phi$ has an analytic 
continuation:

\begin{equation}
\lbl{eq.anPhi}
\Phi: \hat{\BC} \longto \BC.
\end{equation} 
In \cite[Def.1.7]{Ga2} we show that $\Phi$ satisfies three 4-term relations,
one of which is the analytic continuation of \eqref{eq.4term}. The other
two are dictated by the variation of $\Phi$ along the cuts $(1,\infty)$
and $(-\infty,0)$.

Using the three 4-term relations of $\Phi$, we introduce an extended
version $\wb2C$  in \cite[Def.1.7]{Ga2}:

\begin{definition}
\lbl{def.extb}
The {\em extended group} $\wb2C$ is the $\BC$-vector space
generated by the symbols $\la x \ra$ with $x=(z;p,q) \in \hat{\BC}$,
subject to the {\em extended 4-term relation}:
\begin{equation}
\lbl{eq.ex4term}
\la x_0;p_0,q_0 \ra 
- \la x_1;p_1,q_1 \ra 
+(1-x_0) \la \frac{x_1}{1-x_0};p_2,q_2 \ra
-(1-x_1) \la  \frac{x_0}{1-x_1};p_3,q_3 \ra =0
\end{equation}
for 
$ ((x_0;p_0,q_0), \dots, (x_3;p_3,q_3)) \in \hatFT$, and
the relations:
\begin{eqnarray}
\lbl{eq.transfer1}
\la x;p,q \ra -\la x;p,q' \ra &=& \la x;p,q-2 \ra- \la x;p,q'-2 \ra 
\\
\lbl{eq.transfer2}
\la x;p,q \ra -\la x;p',q \ra &=& \la x;p-2,q \ra-\la x;p'-2,q \ra 
\end{eqnarray}
for $x \in \BC^{**}$, $p,q,p',q' \in 2 \BZ$.
\end{definition}

Since the three 4-term relations in the definition of $\wb2C$ are satisfied
by the entropy function, it follows that $\Phi$ gives rise to a 
{\em regulator map}:

\begin{equation}
\lbl{eq.R}
R: \wb2C \longto \BC.
\end{equation}

For a motivation of the extended group $\wb2C$ and its relation to
additive (i.e., infinitesimal) $K$-theory and 
infinitesimal polylogarithms, see \cite[Sec.1.1]{Ga2} and references therein.

\subsection{\Gterm s and additive $K$-theory}
\lbl{sub.gtermK}

In this section, it will be more convenient to use the presentation 
\eqref{eq.tnk3} of \gterm s. In this case, we have:

\begin{equation}
\lbl{eq.newtnk}
t_{n,k}=C_0^n \prod_{i=1}^r C_i^{k_i}
 \prod_{j=1}^J B_j(n,k)!^{\e_j} C_j(n,k)!^{-\e_j}
(B_j-C_j)(n,k)!^{-\e_j}.
\end{equation}
The Stirling formula motivates the constructions in this section. 
Indeed, we have the following:

\begin{lemma}
\lbl{lem.stirling}
For $a>b>0$, we  have:
\begin{equation}
\lbl{eq.phiab}
a \Phi\left(\frac{b}{a}\right) = a \log(a)-b \log(b)-(a-b)\log(a-b).
\end{equation}
and
\begin{equation}
\lbl{eq.binomiala}
\binom{an}{bn} \sim e^{n a \Phi(\frac{b}{a})} 
\sqrt{\frac{a}{2b(a-b) \pi n}} \left(1+ O\left(\frac{1}{n}\right)\right).
\end{equation}
\end{lemma}

\begin{proof}
Equation \eqref{eq.phiab} is elementary. Equation \eqref{eq.binomiala}
follows from the Stirling formula \eqref{eq.stirling}.
\end{proof}

The next lemma gives a combinatorial proof of the
4-term relation of the entropy function.

\begin{lemma}
\lbl{lem.4term}
For $a,b,a+b \in (0,1)$, $\Phi$ satisfies the 4-term relation:
\begin{equation}
\lbl{eq.4term}
\Phi(b)-\Phi(a)+(1-b) \Phi\left(\frac{a}{1-b}\right) -
(1-a) \Phi\left(\frac{b}{1-a}\right)=0.
\end{equation}  
\end{lemma}

\begin{proof}
The 4-term relation follows from the associativity of the multinomial 
coefficients

\begin{equation}
\lbl{eq.assocbinom}
\binom{\a+\b+\ga}{\a} \binom{\b+\ga}{\b}=
\binom{\a+\b+\ga}{\b} \binom{\a+\ga}{\a}=
\frac{(\a+\b+\ga)!}{\a!\, \b! \, \ga!}
\end{equation}
applied to $(\a,\b,\ga)=(an,bn,cn)$
and Lemma \ref{lem.stirling}, and the specialization to $a+b+c=1$. 
In fact, the 4-term relation \eqref{eq.4term} and a local integrability
assumption uniquely
determines $\Phi$ up to multiplication by a complex number. See for example,
\cite{Da} and \cite[Sec.5.4,p.66]{AD}. 
\end{proof}

\begin{corollary}
\lbl{cor.newgterm}
If $\ft$ is a \gterm\ as in \eqref{eq.tnk3}, then its potential
function is given by:
\begin{equation}
\lbl{eq.potV2}
V_{\ft}(w)=C(w)+\sum_{j=1}^J \e_j B_j(w) \Phi\left(
\frac{C_j(w)}{B_j(w)}\right).
\end{equation}
\end{corollary}
Consider the complement $\BC^{r}\setminus\calA'_{\ft}$ of the 
linear hyperplane arrangement given by:

\begin{equation}
\lbl{eq.arrp}
\calA'_{\ft}=\{ w\in \BC^{r} \,| \, \prod_{j=1}^J B_j(w) C_j(w) (B_j(w)-1)
(B_j(w)-C_j(w)) =0 \}.
\end{equation}
Let 

\begin{equation}
\lbl{eq.varpi}
\varpi: \hat{\BC}^{r} \longto \BC^{r}\setminus\calA'_{\ft}
\end{equation}
denote the universal abelian cover of $\BC^{r}\setminus\calA'_{\ft}$. 
For $j=1,\dots, r$, the functions $B_j(w)$, $C_j(w)/B_j(w)$ have analytic
continuation:
$$
B_j: \hat{\BC}^{r} \longto \BC, \qquad 
\frac{C_j}{B_j}: \hat{\BC}^{r} \longto \hat{\BC}
$$
It follows that the potential function given by \eqref{eq.potV2}
has an analytic continuation:

\begin{equation}
\lbl{eq.Vhat}
V_{\ft}: \hat{\BC}^{r} \longto \BC.
\end{equation}
Let $\calC_{\ft}$ denote the set of critical points of $V_{\ft}$.

\begin{definition}
\lbl{def.betaft}
For a \gterm\ $\ft$ as in \eqref{eq.tnk3}, 
we define the map:
\begin{equation}
\lbl{eq.betaft}
\b_{\ft}: \calC_{\ft} \longto \wb2C
\end{equation}
by:
\begin{equation*}
\b_{\ft}(w)=\sum_{j=1}^J \e_j B_j(w) \la \frac{C_j(w)}{B_j(w)} \ra
\end{equation*}
\end{definition}

\begin{theorem}
\lbl{thm.K}
For every \gterm\ $\ft$, we have a commutative diagram:
$$
\begin{diagram}
\node{\calC_{\ft}}
\arrow{e,t}{\b_{\ft}}
\arrow{s,r}{\varpi}
\node{\wb2C }
\arrow{s,r}{e^{R}} \\
\node{X_{\ft}}
\arrow{e,t}{C_0 e^{-V_{\ft}}}
\node{\CV_{\ft}}
\end{diagram}
$$
\end{theorem}

\subsection{Proof of Theorem \ref{thm.K}}
\lbl{sub.thmK}

We begin by observing that the analytic continuation of the logarithm function
gives an analytic function:
$$
\Log: \hat{\BC} \longto \BC.
$$
The proof of part (a) of Theorem \ref{thm.critV} 
implies that $w \in \calC_{\ft}$ if and only if $w$ satisfies the
Logarithmic Variational Equations:

\begin{equation}
\lbl{eq.varlog3}
\log C_i +
\sum_{j=1}^J \e_j( v_i(B_j) \Log(B_j(w)) - v_i(C_j) \Log(C_j(w))   
- v_i(B_j-C_j) \Log((B_j-C_j)(w)))=0 
\end{equation}
for $i=1,\dots,r$. Exponentiating, this implies that 
$\calC_{\ft}=\varpi^{-1}(X_{\ft})$.

The proof of part (b) of Theorem \ref{thm.critV} implies that if $w
\in \calC_{\ft}$, then:

\begin{equation*}
V_{\ft}(w) = \log C_0 +
\sum_{j=1}^J \e_j ( v_0(B_j) \Log(B_j(w)) -  v_0(C_j) \Log(C_j(w))   
-  v_0(B_j-C_j) \Log((B_j-C_j)(w))).
\end{equation*}

On the other hand, if $w \in \calC_{\ft}$, we have:

\begin{eqnarray*}
R(\b_{\ft}(w)) &=& -\sum_{j=1}^J \e_j B_j(w) \Phi\left(\frac{C_j(w)}{B_j(w)} 
\right) \\
&=&
-\sum_{j=1}^J \e_j (B_j(w) \Log(B_j(w)) -C_j(w) \Log(C_j(w))-
(B_j(w)-C_j(w)) \Log((B_j-C_j)(w))) 
\end{eqnarray*}
where the last equality follows from the analytic continuation of 
\eqref{eq.phiab}.
Expanding the linear forms
$B_j$, $C_j$, and $B_j-C_j$ with respect to the variables $w_i$ for 
$i=0,\dots,r$, and 
using the Logarithmic Variational Equations \eqref{eq.varlog3}, (as
in the proof of part (b) of Theorem \ref{thm.critV}), it follows
that

\begin{equation*}
R(\b_{\ft}(w)) = 
-\sum_{j=1}^J \e_j ( v_0(B_j) \Log(B_j(w)) -  v_0(C_j) \Log(C_j(w))   
-  v_0(B_j-C_j) \Log((B_j-C_j)(w))).
\end{equation*}
Thus,

\begin{equation*}
V_{\ft}(w) = \log C_0 -R(\b_{\ft}(w)).
\end{equation*}
This concludes the proof of Theorem \ref{thm.K}.

\section{Proof of Theorems \ref{thm.kempty} and \ref{thm.1}}
\lbl{sec.pfthm1}

\subsection{Proof of Theorem \ref{thm.kempty}}
\lbl{sub.pfkempty}

A $0$-dimensional \gterm\ is of the form:

$$
\ft_n=C_0^n \prod_{j=1}^J (b_j n)!^{\e_j}
$$
where $b_j \in \BN$, $\e_j=\pm 1$ for $j=1,\dots, J$ satisfying
$ \sum_{j=1}^J \e_j b_j=0$. Since $\BZ^0=\{0\}$, it follows that
$$
a_{\ft,n}=\ft_n
$$
is so-called closed form. The Newton polytope of $\ft$ is given by
$P_{\ft}=\{0\} \subset \BR^0$. In addition, $A_j(w)=b_j$ and $v_0(A_j)=b_j$
for $j=1,\dots,J$. By definition, we have $X_{\ft}=\{0\}$, and 

\begin{equation}
\lbl{eq.CV0dim}
\CV_{\ft}=\{C_0^{-1} \prod_{j=1}^J b_j^{-\e_j b_j} \}.
\end{equation}
On the other hand, $G_{\ft}(z)$ is a {\em hypergeometric series} with
singularities $\{0,\CV_{\ft}\}$; see for example \cite[Sec.5]{O} and \cite{No}.
The result follows.

\subsection{Proof of Theorem \ref{thm.1}}
\lbl{sub.pfthm1}

The proof of Theorem \ref{thm.1} is a variant of Laplace's method and
uses the positivity of the restriction of the potential function to $P_{\ft}$.
See also \cite[Sec.5.1.4]{Kn}. Suppose that $\ft_{n,k} \geq 0$ for all $n,k$.
Recall the corresponding polytope $P_{\ft} \subset \BR^{r}$ and consider
the restriction of the potential function to $P_{\ft}$:
$$
V_{\ft}:P_{\ft} \longto \BR.
$$ 
It is easy to show that $\Phi(x)>0$ for $x \in (0,1)$. This is
illustrated by the plot of the entropy function for $x \in [0,1]$:

$$
\psdraw{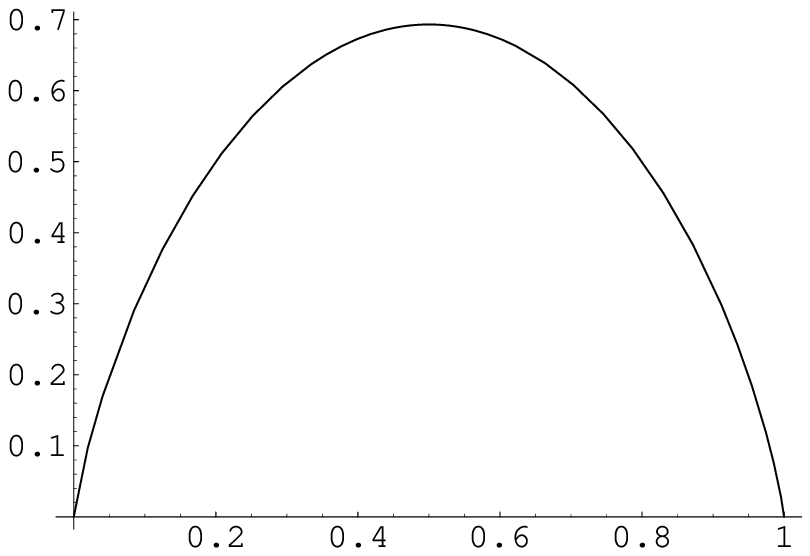}{2in}
$$
It follows that
the restriction of the potential function on $P_{\ft}$ is nonnegative
and continuous. By compactness it follows that
the function achieves a maximum in $\hat{w}$ in the interior of $P_{\ft}$. 
It follows that for every $k \in n P_{\ft} \cap \BZ^r$ we have:
$$
0 \leq \ft_{n,k} \leq \ft_{n, n\hat{w}}
$$
Summing up over the lattice points $k \in n P_{\ft} \cap \BZ^r$, and using
the fact that the number of lattice points in a rational convex polyhedron
(dilated by $n$) is a polynomial function of $n$, it follows 
that there exist a polynomial $p(n) \in \BQ[n]$ so that
for all $n \in \BN$ we have:
$$
\ft_{n,n \hat{w}} \leq a_{\ft,n} \leq p(n) \ft_{n, n \hat{w}}
$$
Using Corollary \ref{cor.app}, 
it follows that there exist polynomials $p_1(n),p_2(n) 
\in \BQ[n]$ so that for all $n \in \BN$ we have:
$$
p_1(n) e^{n V_{\ft}(\hat{w})} \leq a_{\ft,n} \leq p_2(n) e^{n V_{\ft}(\hat{w})}
$$
This implies that $G_{\ft}(z)$ has a singularity
at $z=e^{-V_{\ft}(\hat{w})} >0$. 
Since the maximum lies in the interior of $P_{\ft}$, it follows that 
$\hat{w}$ is a critical point of $V_{\ft}$. Thus, $\hat{w} \in S_{\ft}$
and consequently, $e^{-V_{\ft}(\hat{w})} \in \CV_{\ft}$.

If in addition $\S_{\ft}\setminus\{0\}$ is irreducible over $\BQ$, it follows
that the Galois group $\text{Gal}(\bar\BQ/\BQ)$ acts transitively on $\S_{\ft}$.
This implies that $\S_{\ft} \subset \CV_{\ft}$. 
\qed

\begin{remark}
\lbl{rem.thm1}
When $\ft_{n,k} \geq 0$ for all $n,k$, then $G_{\ft}(z)$ has a singularity at 
$\rho>0$
where $1/\rho$ is the radius of convergence of $G_{\ft}(z)$. This is known as
Pringsheim's theorem, see \cite[Sec.7.21]{Ti}.
\end{remark}

\section{Some examples}
\lbl{sec.examples}

\subsection{A closed form example}
\lbl{sub.exclosed}

As a warm-up example, consider the 1-dimensional \sterm\ 
$$
\ft_{n,k}=\binom{n}{k}=\frac{n!}{k!(n-k)!}
$$
whose corresponding sequence is closed form and the generating series
is a rational function:
$$
a_{\ft,n}=2^n \qquad G_{\ft}(z)=\frac{1}{1-2z}
$$
In that case $\S_{\ft}=\{1/2\}$ and $E_{\ft}=\BQ$.

To compare with our ansatz, the Newton polytope is given by:
$$
P_{\ft}=[0,1] \subset \BR.
$$
The Variational Equations \eqref{eq.vara} are:

\begin{eqnarray*}
\frac{1}{w(1-w)^{-1}}=1
\end{eqnarray*}
in the variable $w=w_1$, with solution set 
$$
X_{\ft}=\{1/2 \}
$$
Thus,
$$
\CV_{\ft}=
\{ \frac{1}{(1-w)^{-1}} \, | \, w=1/2 \}=\{1/2\}.
$$
For the other two faces $\D_0=\{0\}$ and $\D_1=\{1\}$ of the Newton
polytope $P_{\ft}$, the restriction is a $0$-dimensional \gterm . 
Equation \eqref{eq.CV0dim} gives that
$$
\ft|_{\D_0}=\ft_{n,k}|_{k=0}=1, \qquad 
\ft|_{\D_1}=\ft_{n,k}|_{k=n}=1.
$$
Thus,
$$
\CV_{\ft|_{\D_0}}=\CV_{\ft|_{\D_1}}=\{1\}.
$$
This implies that 
$$
S_{\ft} =
\{0,1, 1/2 \}
\qquad
K_{\ft}=\BQ,
$$
confirming Conjecture \ref{conj.1}. For completeness, the potential 
function is given by:

\begin{eqnarray*}
V_{\ft}(w)= \Phi(w).
\end{eqnarray*}

\subsection{The Apery sequence}
\lbl{sub.apery}

As an illustration of Conjecture \ref{conj.1} and Theorem \ref{thm.1}, 
let us consider the \sterm\ 
\begin{equation}
\lbl{eq.tapery}
\ft_{n,k}=\binom{n}{k}^2 \binom{n+k}{k}^2
=\left(\frac{(n+k)!}{k!^2 (n-k)!}\right)^2
\end{equation}
and the corresponding sequence \eqref{eq.aperya}. Equation \eqref{eq.apery3}
implies that the singularities of $G(z)$ are a subset of the roots
of the equation
$$
z^2(z^2-34z+1)=0.
$$
In addition, a nonzero singularity exists. Thus, we obtain that

\begin{equation}
\lbl{eq.aperyS}
\{0,17 +12 \sqrt{2}, 17 -12 \sqrt{2} \}\subset
\S_{\ft}\subset\{0,17 +12 \sqrt{2}, 17 -12 \sqrt{2} \}, \qquad
E_{\ft}=\BQ(\sqrt{2}).
\end{equation}
Thus $E_{\ft}$ is a quadratic number field of type $[2,0]$ 
and discriminant $8$. 

On the other hand, the Newton polytope is given by:
\begin{equation}
\lbl{eq.exap1}
P_{\ft}= [0,1] \subset \BR.
\end{equation}
The Variational Equations \eqref{eq.vara} are:
\begin{eqnarray}
\lbl{eq.aperyw}
\left( \frac{1-w}{w} \frac{1+w}{w} \right)^2=1
\end{eqnarray}
in the variable $w=w_1$, with solution set 
$$
X_{\ft}=\{1/\sqrt{2}, -1/\sqrt{2} \}.
$$
Thus,
$$
\CV_{\ft}=
\{ \left( \frac{1-w}{1+w} \right)^2 \, | \, 
w=\pm 1/\sqrt{2} \}=\{17 +12 \sqrt{2}, 17 -12 \sqrt{2} \}
$$
For the other two faces $\D_0=\{0\}$ and $\D_1=\{1\}$ of the Newton
polytope $P_{\ft}$ the restriction is a $0$-dimensional \gterm :
$$
\ft|_{\D_0}=\ft_{n,k}|_{k=0}=1, \qquad 
\ft|_{\D_1}=\ft_{n,k}|_{k=n}=\left(\frac{(2n)!}{n!^2}\right)^2
$$
Equation \eqref{eq.CV0dim} implies that
$$
\CV_{\ft|_{\D_0}}=\{1\}, \qquad \CV_{\ft|_{\D_1}}=\{16\}.
$$
Therefore, 
$$
S_{\ft} =
\{0,1,16, 17 +12 \sqrt{2}, 17 -12 \sqrt{2} \}
\qquad
K_{\ft}=\BQ(\sqrt{2}),
$$
confirming Conjecture \ref{conj.1}. For completeness, the potential 
function is given by:

\begin{eqnarray*}
V_{\ft}(w)=2 \Phi(w)+2(1+w)\Phi\left(\frac{w}{1+w}\right).
\end{eqnarray*}

\subsection{An example with critical points at the boundary}
\lbl{sub.ex2}

In this example we find critical points at the boundary of the Newton
polytope even though the \gterm\ is positive. This shows that Theorem
\ref{thm.1} is sharp. Consider the \gterm\

\begin{equation}
\lbl{eq.ex2a}
\ft_{n,k}=\frac{1}{\binom{n}{k}}=
\frac{(n-k)!k!}{n!}
\end{equation}
and the corresponding sequence $(a_{\ft,n})$.
The {\tt zb.m} implementation of the WZ algorithm (see \cite{PR1,PR2}) gives
that $(a_{\ft,n})$ satisfies the inhomogeneous recursion relation:

\begin{equation}
\lbl{eq.ex2b}
-2(n+1) a_{n+1}+ (n+2) a_{n}=-2-2n
\end{equation}
for all $n \in \BN$ with initial conditions $a_0=1$. It follows that
$(a_n)$ satisfies the homogeneous recursion relation:
\begin{equation}
\lbl{eq.ex2c}
-2(n+3)a_{n+3}+(12+5n)a_{n+2}-4(n+2)a_{n+1}+(n+2) a_n=0
\end{equation}
for all $n \in \BN$ with initial conditions $a_0=1,a_1=2,a_2=5/2$.
The corresponding generating series $G_{\ft}(z)$ satisfies the inhomogeneous ODE:

\begin{equation}
\lbl{eq.ex2d}
(z-1)^2(z-2) f'(z) +2(z-1)^2 f(z) + 2=0
\end{equation}
with initial conditions $f(0)=1$. 
We can convert \eqref{eq.ex2d} into a homogeneous ODE by differentiating once.
It follows that

\begin{equation}
\lbl{eq.ex2e}
\S_{\ft} \subset\{1,2\}, \qquad E_{\ft}=\BQ.
\end{equation}

On the other hand, $P_{\ft}=[0,1] \subset \BR$.
The Variational Equations  \eqref{eq.vara} are:

\begin{equation*}
\frac{w}{1-w}=1 
\end{equation*}
in the variable $w=w_1$ 
with solution set $X_{\ft}=\{1/2\}$ and 
$$
\CV_{\ft}=\{ \frac{1}{1-w} \, | \, w=\frac{1}{2}\}=
\{2\}.
$$
For the other two faces $\D_0=\{0\}$ and $\D_1=\{1\}$ of the Newton
polytope $P_{\ft}$ we have:
$$
\ft|_{\D_0}=\ft_{n,k}|_{k=0}=1, \qquad 
\ft|_{\D_1}=\ft_{n,k}|_{k=n}=1.
$$
Thus,
$$
\CV_{\ft|_{\D_0}}=\CV_{\ft|_{\D_1}}=\{1\}.
$$
This implies that 
$$
S_{\ft} =
\{0,1,2 \}
\qquad
K_{\ft}=\BQ,
$$
confirming Conjecture \ref{conj.1}. For completeness, the potential 
function is given by:

\begin{eqnarray*}
V_{\ft}(w)=-\Phi(w).
\end{eqnarray*}

\section{Laurent polynomials and \sterm s}
\lbl{sec.laurentp}

\subsection{Proof of Theorem \ref{thm.3}}
\lbl{sub.Pta}

In this section we will prove Theorem \ref{thm.3}. Consider
$F \in \bar\BQ[x_1^{\pm 1},\dots, x_d^{\pm 1}]$.
Let us decompose $F$ into a sum of monomials with coefficients

\begin{equation}
\lbl{eq.monomials}
F=\sum_{\a \in \calA} c_{\a} x^{\a}
\end{equation}
where $\calA$ is the finite set of monomials of $F$, and where for every
$\a=(\a_1,\dots,\a_d)$ we denote $x^{\a}=\prod_{j=1}^d x_j^{\a_j}$.
Let 
\begin{equation}
\lbl{eq.rankF}
r=|\calA|
\end{equation}
denote the number of monomials of $F$.
Recall that
$$
\D_{n}=\{(y_1,\dots,y_{n+1}) \in \BR^{n+1} \, | \, 
\sum_{j=1}^{n+1} y_j=1, y_i \geq 0,
\,\, i=1,\dots, n+1\}
$$ 
denotes the standard $n$-dimensional {\em simplex} in $\BR^{n+1}$.
Part (a) of Theorem \ref{thm.3} follows easily from the multinomial 
coefficient theorem. Indeed, for every $n \in \BN$ we have:
\begin{eqnarray*}
F^n &=&\sum_{\sum_{\a \in\calA} k_{\a}=n} 
\frac{n!}{\prod_{\a \in \calA} k_{\a}!} \prod_{\a \in \calA} c_{\a}^{k_{\a}}
x^{k_{\a} \a} \\
&=& \sum_{\sum_{\a \in\calA} k_{\a}=n} 
\frac{n!}{\prod_{\a \in \calA} k_{\a}!} \prod_{\a \in \calA} c_{\a}^{k_{\a}}
\cdot x^{\sum_{\a \in \calA} k_{\a} \a}
\end{eqnarray*}
It follows that for every $n \in \BN$ we have:

\begin{eqnarray*}
\Tr(F^n) &=& 
\sum_{\sum_{\a \in\calA} k_{\a}=n, \sum_{\a \in \calA} k_{\a}\a=0 } 
\frac{n!}{\prod_{\a \in \calA} k_{\a}!} \prod_{\a \in \calA} c_{\a}^{k_{\a}} \\
&=& \sum_{k \in n P_{\ft_F} \cap \BZ^r} \ft_{F,n,k}
\end{eqnarray*}
where

\begin{equation}
\lbl{eq.ftFnk}
\ft_{F,n,k} = \frac{n!}{\prod_{\a \in \calA} k_{\a}!} \prod_{\a \in \calA} 
c_{\a}^{k_{\a}}
\end{equation}
and the Newton polygon $P_{\ft_F}$ is given by

\begin{equation}
\lbl{eq.PftF}
P_{\ft_F}=\D_{r-1} \cap W
\end{equation}
where
\begin{equation}
\lbl{eq.VF}
W=\{ (x_{\a}) \in \BR^{r} \, | \, \sum_{\a \in \calA} x_{\a} \a=0 \}
\end{equation}
is a linear subspace of $\BR^{r}$. 

To prove part (b) of Theorem \ref{thm.3}, let us assume that the origin is
in the interior of the Newton polytope of $F$. Such $F$ are also called 
{\em convenient} in singularity theory; \cite{Kou}. If $F$ is not 
convenient, we can replace it with its the restriction $F_f$ to a face $f$ of 
its Newton polytope that contains the origin and observe that 
$\Tr((F_f)^n)=\Tr(F^n)$.

When $F$ is convenient, 
it follows that $W$ has dimension $r-d$ and $W \cap C^o \neq 
\emptyset$, where $C^o$ is  
the interior of the cone $C$ which is spanned by the coordinate vectors
in $\BR^{r}$. (b) follows from Lemma \ref{lem.conesimplex} below.

For part (c) of Theorem \ref{thm.3}, fix 
$F \in M_N(\bar\BQ[x_1^{\pm 1},\dots, x_d^{\pm 1}])$, $G_F(z)$.
Recall that for an invertible matrix $A$ we have 
$A^{-1}=\det(A)^{-1} \text{Cof}(A)$, where $\text{Cof}(A)$ is the co-factor
matrix. It follows that

\begin{equation}
\lbl{eq.zF}
\sum_{n=0}^\infty F^n z^n = (I-z F)^{-1} =\frac{1}{\det(I-zF)} \text{Ad}(I-zF)
\end{equation}
where $\text{Ad}(I-zF) \in M_N(\bar\BQ[x_1^{\pm 1},\dots, x_d^{\pm 1},z])$.
Now 

$$
\det(I-zF)=1-\sum_{\a \in \calS} c_{\a} z^{d_{\a}} x^{\a}
$$ 
where $\calS$ is a finite set, $c_{\a} \in \bar\BQ$ and $d_{\a} \in 
\BN\setminus\{0\}$. Thus, 

$$
\frac{1}{\det(I-zF)}=\sum_{n=0}^\infty \sum_{ \sum_{\a \in \calS} k_{\a}=n} \frac{n!}{
\prod_{\a \in\calS} k_{\a}!} z^{\sum_{ \a \in\calS} d_{\a} k_{\a}} \prod_{ \a \in\calS}  
x^{\sum_{\a \in \calS} k_{\a} \a}
$$
Substituting the above into Equation \eqref{eq.zF} and taking the
constant term, it follows that there exists a finite set $S$ and a 
finite collection $\ft^{(j)}_F$ of \gterm s for $j \in S$ such that
for every $n \in \BN$ we have:

$$
\Tr(F^n)=\sum_{j \in S} a_{\ft^{(j)}_{F,n}}
$$

It follows that the moment generating series $G_F(z)$ (defined in
Equation \eqref{eq.GF}) is given by:

$$
G_F(z)=\sum_{j \in S} G_{\ft^{(j)}_{F}}(z).
$$
Since $G_{\ft^{(j)}_{F}}(z)$ is a $G$-function (by Theorem \ref{thm.G}),
and the set of $G$-functions is closed under addition (see \cite{An1}),
this concludes the proof Theorem \ref{thm.3}.
\qed

\begin{remark}
\lbl{rem.affine}
The \gterm s $\ft^{(j)}_F$ in the above proof use affine linear forms, rather 
than linear ones. In other words, they are given by 
\begin{equation}
\lbl{eq.deftermaffine}
\ft_{n,k}=C_0^n \prod_{i=1}^r C_i^{r_i} \prod_{j=1}^J A_j(n,k)!^{\e_j}
\end{equation}
where $C_i$ are algebraic numbers for $i=0,\dots,r$, 
$\e_j=\pm 1$ for $j=1,\dots,J$, and $A_j$ are affine linear forms,
given by $A_j(n,k)=A^{\text{lin}}_j(n,k)+b_j$
where $A^{\text{lin}}_j(n,k)$ are linear forms that satisfy the balance condition 

\begin{equation}
\lbl{eq.Ajsumlin}
\sum_{j=1}^J \e_j A^{\text{lin}}_j=0.
\end{equation}
Theorem \ref{thm.G} remains true for such \gterm s. For an \gterm\ $\ft$
of the form \eqref{eq.deftermaffine} consider the \gterm\ $\ft^{\text{lin}}$ 
defined by:

\begin{equation}
\lbl{eq.deftermaff}
\ft^{\text{lin}}_{n,k}=C_0^n \prod_{i=1}^r C_i^{r_i} \prod_{j=1}^J A^{\text{lin}}_j(n,k)!^{\e_j}
\end{equation}
We define $S^{\ft}=S^{\ft^{\text{lin}}}$. 
\end{remark}

\begin{remark}
\lbl{rem.toric}
In the notation of Theorem \ref{thm.3}, $P_{\ft_F}$ is a {\em simple} 
polytope and the associated toric variety is projective and has {\em quotient 
singularities}; see \cite{Fu}. In other words, the stabilizers of the torus 
action on the toric variety are finite abelian groups.
\end{remark}

The following lemma was communicated to us by J. Pommersheim.

\begin{lemma}
\lbl{lem.conesimplex}
If $W$ is a linear subspace of $\BR^n$ of dimension $s$, and $W \cap  C^o
\neq \emptyset$, where $C$ is the cone spanned by the $n$ coordinate
vectors in $\BR^n$, then 
$\D_{n-1} \cap W$ is a combinatorial simplex in $V$.
\end{lemma}

\begin{proof}
We can prove the claim by downward induction on $s$. When $s=n-1$,
$W$ is a hyperplane. 
Write $D_{n-1}=C \cap H$ where $H$ is the hyperplane
given by $\sum_{i=1}^n x_i=0$. Observe that $C$ is a simplicial cone
and the intersection $C \cap W$ is a simplicial cone in $W$.
Since the intersection of a simplicial cone inside $C$
with $H$ is a simplicial cone,
it follows that $D_{n-1} \cap W=(C \cap W) \cap H$ is a simplicial cone.
This proves the claim when $s=n-1$. A downward induction on $s$ concludes
the proof.
\end{proof}

\subsection{The Newton polytope of a Laurent polynomial and its associated
\gterm }
\lbl{sub.newton}

In a later publication we will study the close relationship between
the Newton polytope of a Laurent polynomial $F$ 
and the Newton polytope of the corresponding \sterm\ $\ft_F$.
In what follows, fix a Laurant polynomial $F$ as in \eqref{eq.monomials}
and its associated \gterm\ $\ft_F$ as in \eqref{eq.ftFnk}.
With the help of the commutative diagram of Proposition \ref{prop.Psi}
below, we will compare the
extended critical values of $F$ with the set $S_{\ft_F}$. Keep in mind
that \cite{DvK} use the Newton polytope of $F$, whereas our ansatz
uses the Newton polytope of $\ft_F$.

Consider the polynomial map

\begin{equation}
\lbl{eq.phi}
\phi: \bar\BQ[w^{\pm 1}_{\a} | \a \in \calA] \longto 
\bar\BQ[x_1^{\pm 1},\dots, x_d^{\pm 1}]
\end{equation}
given by $\phi(w_{\a})=x^{\a}$. Its kernel $\ker(\phi)$ is a 
{\em monomial ideal}. Adjoin the monomial $w_{\a_0}$ to
$\calA$ (if it not already there), where $\a_0=0 \in \BZ^d$, and consider
the homogenous ideal $\ker^h(\phi)$, where the degree of $w_{\a}$
is $\a$. 

\begin{lemma}
\lbl{lem.vara-alt}
The Variational Equations \eqref{eq.vara} for the \gterm\ $\ft_F$
are equivalent to the following system of equations:
\begin{equation}
\lbl{eq.vara-alt}
\begin{cases}
\prod_{\a \in \calA} x_{\a}^{-p_{\a}} c_{\a}^{p_{\a}}& =1  \,\, 
\text{for} \,\, (w_{\a}^{p_{\a}}) \in \ker^h(\phi)
\\
\sum_{\a \in \calA} x_{\a} \a &=0  \\
\sum_{\a \in \calA} x_{\a} & =1.
\end{cases}
\end{equation}
in the variables $(x_{\a})_{\a \in \calA}$.
\end{lemma}

\begin{proof}
This follows easily from the description of the Newton polytope $P_{\ft_F}$
and the shape of the \gterm\ $\ft$.
\end{proof}

Consider the rational function 

\begin{equation}
\lbl{eq.Psi}
\Psi_F: (\BC^*)^d\setminus F^{-1}(0) \longto (\BC^*)^{r}, \qquad
u \mapsto \Psi_F(u)=\left(\frac{c_{\a}u^{\a}}{F(u)}\right)_{\a \in \calA}
\end{equation}
Observe that the image of $\Psi_F$ lies in the complex affine simplex
$$
\{ (x_{\a}) \in (\BC^*)^{r} \, | \, \sum_{\a \in \calA} x_{\a}=1
\}.
$$

\begin{proposition}
\lbl{prop.Psi}
\rm{(a)} If $u=(u_1,\dots,u_d)$ is a critical point of the restriction
of $F$ on $(\BC^*)^d$
with nonvanishing critical value, then $\Psi_F(u)$ satisfies the Variational
Equations \eqref{eq.vara} for the maximal face $P_{\ft_F}$ of the Newton
polytope $P_{\ft_F}$ of $\ft_P$.
\newline
\rm{(b)} Restricting to those critical points, 
we have a commutative diagram:
\begin{equation}
\lbl{eq.cd}
\begin{diagram}
\node{\text{Critical points of $F$ on $(\BC^*)^d$}} 
\arrow{e,t}{\Psi_F}
\arrow{s,r}{F}
\node{X_{\ft_F,P_{\ft_F}}}
\arrow{s,r}{\text{Map \eqref{eq.StD}}} \\
\node{\text{Critical values of $F$}} 
\arrow{t,e}
\node{S_{\ft_F,P_{\ft_F}}}
\end{diagram}
\end{equation}
\rm{(c)} The top horizontal map of the above diagram is 1-1 and onto.
\end{proposition}

\begin{proof}
For (a) we will use the alternative system of Variational Equations
given in \eqref{eq.vara-alt}. Suppose that $u=(u_1,\dots,u_d)$ is a critical
point of $F$ with nonzero critical value, and let $(w_{\a})_{\a \in \calA}$
denote the tuple $(c_{\a} u^{\a}/F(u))_{\a \in \calA}$. We need to show that
$(w_{\a})_{\a \in \calA}$ is a solution to Equations \eqref{eq.vara-alt}.
Consider a point $(p_{\a})_{\a \in \calA} \in P_{\ft_F}$.
Then, we can verify the 
first line of Equation \eqref{eq.vara-alt} as follows:

\begin{eqnarray*}
\prod_{\a \in \calA} w_{\a}^{-p_{\a}} c_{\a}^{p_{\a}} &=&
\prod_{\a \in \calA} \left(\frac{c_{\a} u^{\a}}{F(u)}\right)^{-p_{\a}} c_{\a}^{p_{\a}}
\\ 
&=&
u^{-\sum_{\a \in \calA} p_{\a} \a} F(u)^{-\sum_{\a \in \calA} p_{\a}} \\
&=& 1
\end{eqnarray*}
since $(w_{\a}^{p_{\a}})_{\a \in \calA}$ in in the kernel of $\phi$.
To verify the second line of Equation \eqref{eq.vara-alt}, recall that
$$
z \pt_z (z^k)=k z^k.
$$ 
Since $u$ is a critical point of $F$, it follows that for every $i=1,\dots,d$
we have:
$$
u_i \pt_{u_i} F(u)=\sum_{\a \in \calA} v_i(\a) c_{\a} u^{\a}=0
$$
where $v_i(\a)$ is the $i$th coordinate of $\a$. It follows that
$\sum_{\a \in \calA} w_{\a} \a=0$, which verifies the second line of 
\eqref{eq.vara-alt}. To verify the third line, we compute:

$$
\sum_{\a \in \calA} \frac{c_{\a} u^{\a}}{F(u)}=\frac{1}{F(u)} 
 \sum_{\a \in \calA} c_{\a} u^{\a} =\frac{F(u)}{F(u)}=1.
$$
(b) and (c) follow from similar computations.
\end{proof}

Let us end this section with an example that illustrates Lemma 
\ref{lem.vara-alt} and Proposition \ref{prop.Psi}.

\begin{example}
\lbl{ex.F}
Consider the Laurent polynomial
$$
F(x_1,x_2)= a x_1 + b x_1^{-1} + c x_2 + d x_2^{-1} + f x_1 x_2 + g
$$
Its Newton polytope is given by 

$$
\psdraw{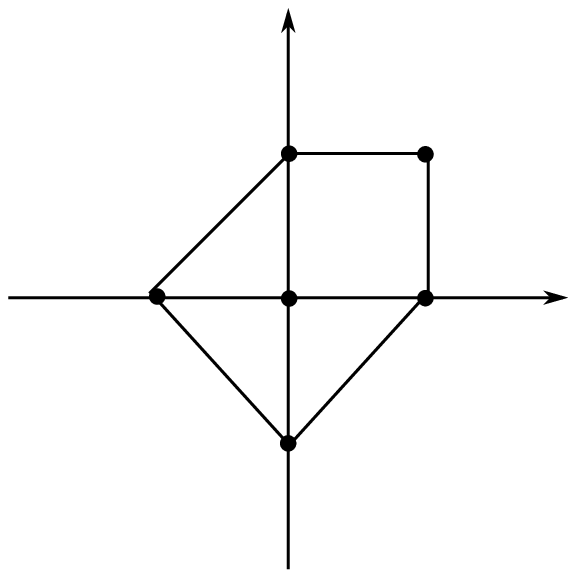}{1.5in}
$$

With the notation from the previous section, we have $d=2$, $r=6$. Moreover,
for every $n \in \BN$, we have

$$
F^n=
\sum_{k_1+\dots+k_6=n} \frac{n!}{k_1!\dots k_6!}
a^{k_1} b^{k_2} c^{k_3} d^{k_4} f^{k_5} g^{k_6} x_1^{k_1-k_2+k_5}
x_2^{k_3-k_4+k_5}.
$$ 
It follows that

$$
\Tr(F^n)=
\sum_{(k_1,\dots,k_6) \in n P_{\ft_F} \cap \BZ^6}
 \frac{n!}{k_1!\dots k_6!}
a^{k_1} b^{k_2} c^{k_3} d^{k_4} f^{k_5} g^{k_6} 
$$
where the Newton polytope $P_{\ft_F} \subset \BR^6$ is given by

\begin{equation}
\lbl{eq.np1}
P_{\ft_F}=\{ (x_1,\dots,x_6) \in \BR^6 \, | \, 
x_1-x_2+x_5=0, \quad x_3-x_4+x_5=0, \quad x_1+\dots+x_6=1, \quad x_i \geq 0
\}
\end{equation}
We can use the variables $(x_1,x_3,x_5)$ to parametrize $P_{\ft_F}$ as
follows:

\begin{equation}
\lbl{eq.np2}
P_{\ft_F}=\{ (x_1,x_3,x_5) \in \BR^3 \, | \, 1 \geq 2x_1+2x_3 + 3 x_5, 
\quad x_i \geq 0
\}
\end{equation}
This confirms that $P_{\ft_F}$ is a combinatorial 2-dimensional simplex.

If $(k_1,\dots,k_6) \in P_{\ft_F}$, then $(k_2,k_4,k_6)=(k_1+k_5,k_3+k_5,
n-2k_1-2k_3-3k_5)$, and

$$
\Tr(F^n)=\sum_{2k_1+2k_3+2k_5=n} 
 \frac{n!}{k_1!(k_1+k_5)!k_3!(k_3+k_5)! k_5! (n-2k_1-2k_3-3k_5)!}
\left(\frac{ab}{g^2}\right)^{k_1} 
\left(\frac{cd}{g^2}\right)^{k_3} 
\left(\frac{bdf}{g^3}\right)^{k_5} g^n.
$$
The Variational Equations \eqref{eq.vara} are:

\begin{eqnarray}
\notag
\frac{1}{w_1 (w_1+w_5) (1-2w_1-2w_3-3w_5)^{-2}} \frac{ab}{g^2} &=& 1 \\
\lbl{eq.vex1}
\frac{1}{w_3 (w_3+w_5) (1-2w_1-2w_3-3w_5)^{-2}} \frac{cd}{g^2} &=& 1  \\
\notag
\frac{1}{(w_1+w_5) (w_3+w_5) w_5 (1-2w_1-2w_3-2w_5)^{-3}} \frac{bdf}{g^3} &=& 1
\end{eqnarray}
in the variables $(w_1,w_3,w_5)$. Reintroducing the variables 
$w_2$, $w_4$ and $w_6$ defined by
$$
w_2=w_1+w_5, \qquad w_4=w_3+w_5, \qquad w_6=1-2w_1-2w_3-2w_5
$$
the Variational Equations become

\begin{eqnarray}
\notag
w_2 &=& w_1+w_5 \\
\lbl{eq.vex2}
w_4 &= & w_3+w_5 \\
\notag
w_6 &= & 1-2w_1-2w_3-2w_5 \\
\notag
\frac{1}{w_1 w_2 w_6^{-2}} \frac{ab}{g^2} &=& 1 \\
\lbl{eq.vex3}
\frac{1}{w_3 w_4 w_6^{-2}} \frac{cd}{g^2} &=& 1  \\
\notag
\frac{1}{w_2 w_4 w_5 w_6^{-3}} \frac{bdf}{g^3} &=& 1
\end{eqnarray}
On the other hand, $w_{\a_0}=w_6$ and 
$\ker(\phi)$ is generated by the relations 
$w_6=1$, $w_1 w_2=1$, $w_3 w_4 =1$ and $w_2 w_4 w_5=1$ which lead to
the homogenous relations $w_1 w_2 w_6^{-2}=1$, $w_3 w_4 w_6^{-2}=1$
and $w_2 w_4 w_5 w_6^{-3}$ for $\ker^h(\phi)$ 
that appear in the above Variational Equations.
This illustrates Lemma \ref{lem.vara-alt}.

Suppose now that $u=(u_1,u_2)$ is a critical point of $F$. Then, $u$
satisfies the equations:

\begin{eqnarray}
\lbl{eq.vex4}
\frac{\pt F}{\pt u_1} &=& a-\frac{b}{u_1^2}+f u_2=0 \\
\lbl{eq.vex5}
\frac{\pt F}{\pt u_2} &=& c-\frac{d}{u_2^2}+f u_1=0.
\end{eqnarray}
The map $\Psi_F$ is defined by $\Psi_F(u_1,u_2)=(w_1,\dots,w_6)$ where

$$
w_1=\frac{a u_1}{F(u)}, \qquad
w_2=\frac{b }{u_1 F(u)}, \qquad
w_3=\frac{c u_2}{ F(u)}, \qquad
w_4=\frac{d }{u_2 F(u)}, \qquad
w_5=\frac{f u_1 u_2}{ F(u)}, \qquad
w_5=\frac{g }{ F(u)}.
$$
If $u=(u_1,u_2)$ satisfies Equations \eqref{eq.vex4}, \eqref{eq.vex5},
it is easy to see that $w=\Psi_F(u)$ satisfies the Variational Equations
\eqref{eq.vex2} and \eqref{eq.vex3}. For example, we have

$$
w_3+w_5 -w_4 = \frac{c u_2}{F(u)} + \frac{f u_1 u_2}{F(u)} -
\frac{d}{ u_2 F(u)}= \frac{u_2}{F(u)}\left( c + f u_1 -\frac{d}{u_2^2}
\right)=0
$$
and
 
$$
\frac{1}{w_3 w_4 w_6^{-2}} \frac{cd}{g^2} = 
\frac{1}{ \frac{c u_2}{F(u)} \frac{d}{u_2 F(u)} 
\left(\frac{g }{F(u)}\right)^{-2} } \frac{cd}{g^2}=1.
$$
This illustrates Proposition \ref{prop.Psi}.
\end{example}

\ifx\undefined\bysame
        \newcommand{\bysame}{\leavevmode\hbox
to3em{\hrulefill}\,}
\fi

\end{document}